\documentclass[final,onefignum,onetabnum]{siamart190516}

\usepackage{lipsum}
\usepackage{amsfonts}
\usepackage{graphicx}
\usepackage{epstopdf}
\usepackage{algorithmic}
\ifpdf
  \DeclareGraphicsExtensions{.eps,.pdf,.png,.jpg}
\else
  \DeclareGraphicsExtensions{.eps}
\fi

\newsiamremark{remark}{Remark}
\newsiamremark{hypothesis}{Hypothesis}
\crefname{hypothesis}{Hypothesis}{Hypotheses}
\newsiamthm{claim}{Claim}

\headers{Multi Asset Kyle-Back Models}{S. Bose and I. Ekren}

\title{Multidimensional Kyle-Back model with a risk averse informed trader\thanks{This version: November 1st, 2021.
\funding{This work was supported in part by NSF Grant DMS 2007826.}}}

\author{Shreya Bose\thanks{Department of Mathematics, Florida State University, Tallahassee, FL  
  (\email{sb18m@my.fsu.edu}}).
\and Ibrahim Ekren\thanks{Department of Mathematics, Florida State University, Tallahassee, FL 
  (\email{iekren@fsu.edu}).}
}

\usepackage{amsopn}

\ifpdf
\hypersetup{
  pdftitle={Multidimensional Kyle-Back model with a risk averse informed trader},
  pdfauthor={S. Bose and I. Ekren}
}
\fi

\usepackage[colorinlistoftodos]{todonotes}
\usepackage{comment}
\newtheorem{assumption}{Assumption}[section]

\newcommand{\pa}{{\partial}}

\newcommand{\ignore}[1]{}

\DeclareMathOperator*{\argmin}{arg\,min}

\newcommand{\cA}{\mathcal{A}}

\newcommand{\cF}{\mathcal{F}}

\newcommand{\cH}{\mathcal{H}}

\newcommand{\cM}{\mathcal{M}}

\newcommand{\cS}{\mathcal{S}}

\newcommand{\Gam}{{\Gamma}}

\def \a{\alpha}

\def \e{\varepsilon}

\def \1{\mathbf 1}

\def\E{\mathbb{E}}

\def\P{\mathbb{P}}

\def\R{\mathbb{R}}

\newcommand{\comm}[1]{}

\begin{document}

\maketitle

\begin{abstract}
 We study the continuous time Kyle-Back model with a risk averse informed trader. We show that in a market with multiple assets and non-Gaussian prices an equilibrium exists. The equilibrium is constructed by considering a Fokker-Planck equation and a system of partial differential equations that are coupled with an optimal transport type constraint at maturity. 

\end{abstract}



\section{Introduction}

The objective of this paper is to continue the study of Kyle-Back models with a risk averse informed trader using the optimal transport theory based methodology initiated in \cite{cel,bose2020kyle}. In this paper, we study markets with multiple stocks and show the existence of equilibrium under the assumption of log-concavity of the distribution of the fundamental value of the assets. 

Kyle's model, introduced in the seminal work \cite{kyle}, has been one of the most influential models in finance. In its original form, the model shows how an equilibrium between a profit maximizing informed trader and a risk-neutral market maker is reached. In this equilibrium, an informed trader who knows the fundamental value of an asset trades optimally to maximize his expected profit against a market maker. The market maker who has inferior information observes the total demand of the informed trader and noise traders and quotes a rational price for the asset. In this seminal work, the distribution of the fundamental value is assumed to be Gaussian and the pricing rule of the market maker is particularly simple. Indeed, the increments of the price quoted by the market maker is proportional to the total demand received. 

A large number of extensions of this model has been considered in the literature. In particular, the risk neutral single-period Kyle models are studied in \cite{caballe1994imperfect,pasquariello2015strategic,garcia2020multivariate}. In discrete time, the risk aversion of the informed trader is studied in \cite{subrahmanyam1991risk}.
The continuous-time model with non-Gaussian fundamental value was introduced in \cite{ba}. Many extensions of the continuous time model have been studied, see \cite{ba,back2000imperfect,baruch2002insider,back2004information,caldentey2010insider,ccd,anderson2013dynamic,ccetin2018financial,back2018identifying}.

Recently, it was shown in \cite{cel} that there exists a deep connection between optimal transport theory and Kyle-Back models and many financially important quantities can be computed using the Monge-Kantorovich duality\footnote{We refer to \cite{galichon2016optimal,santambrogio2015optimal,v} for a summary of fundamental results in optimal transport theory. }. The connection stems from the fact that in many versions of Kyle's model, the strategy of the informed trader has the inconspicuous trading property as defined in \cite{cho}. This property means that the current trading rate of the informed trader has zero mean conditional to the information of the market maker . Such a property creates distributional constraints on the total demand received by the market maker  at maturity. Thus, a natural pricing rule at maturity for the market maker  is to use the optimal transport map (or Brenier's map) as defined in \cite{b,m}. The strategy of the informed trader can also be interpreted as the Brownian bridge whose final condition is determined by an optimal transport map and the private information of the informed trader. Additionally, the potential functions, which are the optimizers of the dual formulation of the optimal transport problem, provides expressions for the profits of the informed trader and the market maker, see \cite[Theorem 3.2 and Corollary 3.1]{cel}. 

In this paper, we aim to extend the results in \cite{bose2020kyle} to a multi-dimensional case. A fundamental challenge in treating the case of the risk-averse insider is the fact that the total demand $Y_t$ is not a sufficient statistic for the pricing rule of the market maker and the strategy of the informed trader. This point was noted in \cite{cho}, where the author shows that both the pricing rule and the strategy of the informed trader are path-dependent functionals of $Y$, see also \cite{cd}. It is also shown in \cite{cho} that the relevant statistic is a stochastic integral $\xi_t=\int_0^t \lambda_s dY_s$ driven by $Y$ where the integrand $\lambda_t$, called the price pressure, is a deterministic function of time. The price pressure is then determined by an algebraic equation. Given the fact that the price pressure is taken to be deterministic, it is shown in \cite{cho} that an equilibrium can only exist if the fundamental value of the asset is Gaussian. 

In \cite{bose2020kyle}, in order to relax the Gaussianity assumption of the fundamental value, the authors allow the price pressure to be a function of $(\xi_t)$. Thus, $(\xi_t)$ solves a stochastic differential equation driven by $Y$. Then, the algebraic condition identifying the price pressure was shown to be an optimal transport type condition between the law of $\xi_T$ in equilibrium and the belief of the market maker on the fundamental value. This condition yields, in the context of \cite{bose2020kyle}, to the existence of a fixed-point between a Fokker-Planck equation and a quasilinear parabolic partial differential equation. This fixed point allows the construction of an equilibrium. 

The multi-asset framework studied in this paper brings forward major challenges. First of all, although the quasilinear equation describing the pricing rule in \cite{bose2020kyle} can be stated in this multi-asset framework, we are not able to construct other functions needed to carry out the existence proof. Thus, we derive novel tools to construct the price pressure function. The main advantage of this methodology is to completely avoid the study of the quasilinear pricing equation. The solution of this equation is readily provided from a first order optimality condition of a convex conjugation. 

Secondly, in our multi-asset framework the optimal transport map is not explicit and we need to prove the continuous dependence of the transport maps as a function of marginals of the transport problem. Given the fact that the PDE methods in \cite{bose2020kyle} require strong regularity assumptions, it was conjectured that the existence of an equilibrium for the multi-asset problem studied here would require strong continuous dependence estimates of the transport maps. However, our novel optimization theory based construction presented here
do not require stringent condition on the functions defined. Thus, we are able to obtain the continuous dependence estimates needed with tools available in the literature, see \cite[Corollary 5.23]{v}. 

Our optimal transport theory based construction also grandly simplifies the filtering problem of the market maker. Indeed, we show that in the equilibrium we construct, the solution of the Kushner's equation associated to the filtering problem of the market maker is solved by the transition density of $(\xi_t)$ in equilibrium. In fact, the optimal transport type constraints defining the fixed-point are exactly the condition we need so that this property of the filtering problem holds. 

The paper is organized as follows. We introduce the problem of interest in Section \ref{s:problem}. After providing preliminary results in Section \ref{s:prem}, we state our main theorem in terms of existence of equilibrium in Section \ref{s:main}. In section \ref{s.ex}, we discuss properties of the equilibrium and compute explicitly the equilibrium for Gaussian fundamental value. In Section \ref{s.proof}, we provide the proofs of the results of the previous sections. 

\subsection{Notations}

Throughout this paper we will use the following notations. We denote by $\cS^n$ (resp. $\cS^n_{>0}$, $\cS^n_{\geq0}$) the set of $n$-dimensional symmetric (resp. symmetric positive, symmetric non-negative) matrices. The components of the vector valued functions will be denoted using superscripts, while a point $x \in\R^n =\R^{n\times 1}$ has its components written with subscripts.
For vector valued functions $u:\R^n\mapsto\R^m, u=(u^1,...,u^m)^\top$, we write the Jacobian matrix as
\[
D_xu = \begin{pmatrix} 
    \frac{\pa u^1}{\pa x_1} & \dots &  \frac{\pa u^1}{\pa x_n}  \\
    & \ddots & \\
    \frac{\pa u^m}{\pa x_1} & \dots &  \frac{\pa u^m}{\pa x_n}
    \end{pmatrix}= \begin{pmatrix} 
   Du^1 \\
   \vdots \\
    Du^m
    \end{pmatrix}\in \R^{m\times n}.
\]

Note that with this convention, in the case $m=1$, $D_xu=\left(\frac{\pa u}{\pa x_1},...,\frac{\pa u}{\pa x_n}\right)\in \R^{1\times n}$ denotes a row vector which is the transpose of its gradient. In the case $m=1$, the Hessian matrix $D^2_xu\in \cS^n$, is given by
\[
D^2_xu = \begin{pmatrix} 
    \frac{\pa^2u}{\pa x_1^2} & \dots &  \frac{\pa^2u}{\pa x_1\pa x_n}  \\
    & \ddots & \\
    \frac{\pa^2u}{\pa x_n \pa x_1} & \dots &  \frac{\pa^2u}{\pa x_n^2}
    \end{pmatrix}
\]
For simplicity of notation, we use $D$ and $D^2$ instead of $D_x$ and $D^2_x$ respectively, whenever there is no risk of confusion.
For any $A\in \R^{n\times n}$, we denote $\underline \lambda(A)$ and $\overline \lambda(A)$ respectively the smallest and largest eigenvalues of $A$. By $tr(A)$, we denote the trace of $A$ given by
\[ tr(A)=\sum_{i=1}^{n}a_{ii}.\]

Let $\Lambda=[0,T]\times C([0,T],\R^n)$ denote the set of $n$ dimensional continuous paths. We endow $\Lambda $ with the pseudo-metric $$d_\infty((t,y_\cdot),(s,\tilde y_\cdot))=|t-s|+\sup_{0\leq r\leq T}|y_{t\wedge r}-\tilde y_{s\wedge r}|.$$

Similarly to \cite{dup,cont,etz}, we say that a mapping $u:\Lambda \to \R$ is $C^{1,2}(\Lambda)$ if 
there exists three continuous mappings $\pa_t u: \Lambda\mapsto \R$, $\pa_y u:\Lambda\mapsto \R^n$, $\pa_{yy} u:\Lambda \mapsto \cS_n$ so that 
for any continuous semi-martingale $Y$ with bounded characteristics, we have the decomposition
$$du(t,Y_\cdot)=\left(\pa_t u(t,Y_\cdot)+\frac{1}{2}tr\left(\pa_{yy}u(t,Y_\cdot)\frac{d\langle Y\rangle_t}{dt}\right)\right)dt+\pa_y u^\top (t,Y_\cdot)dY_t.$$
Note that any functional which is continuous with respect to $d_\infty$ is non-anticipative. 
\section{The multivariate model setup}\label{s:problem}
Let us consider a filtered probability space $(\Omega, \cF,(\cF_t)_{t\in [0,T]},\P)$ which satisfies the usual conditions of right continuity and completeness. We assume that the probability space carries a random vector $\tilde v\in \R^n$ and a standard $n$-dimensional Brownian motion $B$. We assume that $\cF$ is the augmentation of the filtration generated by the process $t\mapsto \tilde v+B_t$ so that $\tilde v$ and $B$ are independent. 

In the economy that we consider, there are $n\geq 1$ stocks that can be traded continuously over the time interval $[0,T]$. There are three market participants namely an informed trader, noise traders and market makers. At the initial time $(t=0)$, the informed trader learns the fundamental value of the $n$ assets at the terminal time $(t=T)$, which we denote by $\tilde v \in \R^n$. 
At the terminal time, this private information becomes available to the other market participants. The noise traders trade in the $n$ assets because of some exogenous needs. The position of the informed trader at each time $t$ is denoted by $X_t \in \R^n$ and that of the noise traders is assumed to have the form $Z_t =\sigma B_t\in \R^n$, where $\sigma\in \cS^n_{>0}$ is a deterministic symmetric positive matrix. 
 
The market makers only observe the total demand $Y_t \in \R^n$ where $Y_t=X_t+Z_t$ without being able to observe the separate trades from $X_t$ and $Z_t$. They have an apriori belief on the distribution of the $n-$dimensional random variable $\tilde v$, which we denote as $\nu$. Based on the information set $(\cF_t^Y)$ of the market maker, in equilibrium, the risk neutral market maker quotes the following rational pricing rule for the $n$ stocks
\begin{align}
  P_t=\E[\tilde v|\mathcal{F}^{Y}_t]  , t \in [0,T].
\end{align}
The main objective of the informed trader is to maximize the expected profit by choosing his trading strategy $X$ optimally. As shown in \cite{cel,b}, the wealth of the informed trader at final time T is given by
\begin{align}\label{eq:wealth}
   W_T=W_T(X,H):=\int_0^T(\tilde v-H(t,X_\cdot+Z_\cdot))^\top dX_t-\sum_{i=1}^n \langle X^i,H^i(t,X_\cdot+Z_\cdot)\rangle_T.
\end{align}
We assume that the informed trader is risk averse with exponential utility 
$$-\gamma \exp(-\gamma W_T)$$ where $\gamma>0$ is her risk aversion parameter.

We make the following standing assumptions on the data of our problem. 
\begin{assumption}\label{assum:main}
The distribution $\nu$ of $\tilde v$ is $\kappa$-strongly log-concave for some $\kappa>0$, i.e. there exists a function $V$ with $D^2V\geq \kappa I_n$ such that the density of $\nu$ with respect to the Lebesgue measure is $x\in \R^n\mapsto e^{-V(x)}$. 
\end{assumption}

Before going into the definition of the equilibrium, we first define the pricing rule of the market maker followed by the class of admissible strategies of the informed trader.
As studied in \cite{cn,cho,bose2020kyle}, in equilibrium the pricing rule of the market maker will be path-dependent.  
\begin{definition}\label{def:pricing}
A pricing rule is a mapping $H:\Lambda\mapsto \R^n$ so that $H^i\in C^{1,2}(\Lambda)$ for all $i=1,\ldots, n$ and satisfies the integrability condition
\begin{align}
\E[|H(T,Z_\cdot)|^2]+\int_0^T \E[|H(t,Z_\cdot)|^2]dt<\infty.
\end{align}
We denote the set of pricing rules by $\cH$.
\end{definition}

\begin{definition}\label{def:strategy} For any given pricing rule $H \in \cH$, a trading strategy for the informed trader is a continuous integrable semi-martingale $X$ with $X_0=0$, adapted to the filtration $\cF$, and satisfying 
\begin{align}\label{eq:addx}
 \E\left[e^{\frac{\gamma^2}{2}\int_0^T |\sigma (\tilde v-H(t,X_\cdot+Z_\cdot))|^2dt}\right]<\infty
\end{align}
Given $H \in \cH$, we denote the set of admissible strategies of the informed trader by $\cA(H)$.

\end{definition}

The objective of the paper is to establish the existence of an equilibrium defined as below. 
\begin{definition}A pricing rule $H^*\in \cH$ of the market maker and a trading strategy $X^*\in \cA(H^*)$ for the informed trader is said to form an equilibrium if the following two conditions hold.
\begin{itemize}
    \item Rationality of the pricing rule: If the informed trader chooses the trading strategy $X^*$, then $P_t=H^*(t,Y_\cdot)=\E[\tilde v|\mathcal{F}^{Y}_t]$ and $P_T = \tilde v$ a.s.
\item Profit maximization condition:  If the market maker's pricing rule is $P_t=H^*(t,Y_\cdot)$, then $X^*$ is a maximizer of 
\begin{align}\label{def:opt}
    \sup_{X \in \cA(H^*) }\E\left[-\gamma \exp(-\gamma W_T(X,H^*))\right].
\end{align}
\end{itemize}
\end{definition}

\section{Preliminary Results}\label{s:prem}
Denote by 
$$C_l=\{\phi\in C^2(\R^n;\R), 0\leq D^2 \phi<l Id,\, \phi(0)=0\}$$ the set of convex functions with second derivatives bounded by $l Id$. 
Let $0\leq l<\frac{1}{\overline\lambda^2( \sigma)  \gamma T}$ and $\phi\in C_{l}$ be fixed and define on $[0,T]\times \R^n$ the functions $(E,\chi,\Gamma,P)=(E^\phi,\chi^\phi,\Gamma^\phi,P^\phi)$ by
\begin{align}\label{eq:defe}
    E(t,z)&:=\frac{1}{\gamma} \ln\E\left[\exp\left(\gamma \phi(z+\sigma(B_T-B_t))\right)\right]\in\R\\
    \chi(t,\xi)&:=\argmin_{z\in \R^n} \left\{E(t,z)+\frac{1}{2\gamma(T-t)}|\sigma^{-1}(\xi-z)|^2\right\}\in \R^n\label{eq:defchi}\\
\Gamma(t,\xi)&:= E(t,\chi(t,\xi))=\min_{z\in \R^n} \left\{E(t,z)+\frac{1}{2\gamma(T-t)}|\sigma^{-1}(\xi-z)|^2\right\}\in \R\label{eq:defgamma}
\\
P(t,\xi)&:= (D E)^\top(t,\chi(t,\xi))\in \R^{ n}.\label{eq:defp}
\end{align}
The following lemma shows that these functions are well-defined and provides their properties that we will need to construct an equilibrium.
\begin{lemma}\label{lem:prep}
For $0\leq l<\frac{1}{\overline\lambda^2( \sigma)  \gamma T}$, and $\phi\in C_l$, the function defined by 
$$(t,\xi,z)\in [0,T]\times \R^n\times \R^n\mapsto E(t,z)+\frac{1}{2\gamma(T-t)}|\sigma^{-1}(\xi-z)|^2$$
is strongly convex in $z$ for all $t<T$, and $\chi(t,\xi)$ is its unique minimizer in $z$ for all $(t,\xi)\in [0,T]\times \R^n$.
The functions $(\chi,\Gamma,P)$ 
are $C^{1,2}([0,T)\times \R^n)\cap C^{0}([0,T]\times \R^n)$ solutions of 
\begin{align}
    \frac{\pa \Gamma}{\pa t}(t,\xi)+\frac{1}{2}tr\Big(\sigma^\top (D\chi)^{-\top}(t,\xi)D^2 \Gamma(t,\xi)(D\chi)^{-1}(t,\xi)\sigma\Big)&=\frac{\gamma |\sigma P(t,\xi)|^2}{2}\label{eq:system1}\\
 \frac{\pa \chi^i}{\pa t}(t,\xi)+\frac{1}{2}tr\Big(\sigma^\top (D\chi)^{-\top}(t,\xi)D^2 \chi^{i}(t,\xi)(D\chi)^{-1}(t,\xi)\sigma\Big)&={\gamma\Big(\sigma^2 P(t,\xi)\Big)^i}\label{eq:system2}\\
 \frac{\pa P^i}{\pa t}(t,\xi)+\frac{1}{2}tr\Big(\sigma^\top (D\chi)^{-\top}(t,\xi)D^2 P^i(t,\xi)(D\chi)^{-1}(t,\xi)\sigma\Big)&=0\label{eq:pdep}.
\end{align}
for all $i=1,\ldots, n$, with final conditions
\begin{align}
    P(T,\xi)={(D\phi(\xi))^\top},\, \chi(T,\xi)=\xi,\, \Gamma(T,\xi)=\phi(\xi).
\end{align}
Additionally, let $\phi_k,\phi\in C_l$ be such that 
\begin{align}\label{eq:cvcl}
\phi_k\to \phi\mbox{ uniformly on compact sets of }\R^n \mbox{ as }k\to \infty.
\end{align}
Then, 
\begin{align}\label{eq:contd}
(\Gamma^{\phi_k}(0,0),\chi^{\phi_k}(0,0))\to(\Gamma^{\phi}(0,0),\chi^{\phi}(0,0)).
\end{align}
\end{lemma}
\begin{remark}
(i) Using the first order optimality condition
\begin{align}\label{eq:chip}
 \chi(t,\xi)=\xi-(T-t)\gamma \sigma^2 P(t,\xi) 
\end{align}
of the optimization problem \eqref{eq:defchi}, one can eliminate $D\chi$ from the equation \eqref{eq:pdep} so that $P$ solves a quasilinear parabolic equation. 
In the one dimensional case of \cite{bose2020kyle}, the function $P$ was directly constructed as the solution of such an equation, see \cite[Equation (3.1)]{bose2020kyle}.
 After constructing $P$, one can obtain $\Gamma$ and $\chi$ by simple integration of $P$. However, we are not able to extend this methodology to the multi-dimensional case of this paper where in order to be able to integrate $P$, we need some symmetry properties of $D\chi$ that we are unable to check. 
 
(ii) The novel multi-dimensional construction presented here also removes strong regularity requirement on the class $C_l$ imposed in \cite{bose2020kyle}. Because of these strong regularity requirements, in \cite{bose2020kyle}, it was conjectured that the extension of \cite{bose2020kyle} to multi-dimension would require the proof of strong continuous dependence estimates for the Monge-Ampere equation. However, by removing these regularity requirements, our novel methodology is able to use simple continuous dependence estimates of optimal transport plans to construct the equilibrium. 

\end{remark}

As mentioned, in the equilibrium we construct, $Y_t$ is not a sufficient statistics to define the strategies of both agents. We introduce below a novel state process $(\xi_t)$ allowing us to define these strategies.  
The next Lemma provides a solution of the Fokker-Planck equation associated with this state variable in equilibrium. 
\begin{lemma}\label{lem:fp}
For $0\leq l<\frac{1}{\overline\lambda^2( \sigma)  \gamma T}$, let $(\xi^0_t)_{t\in [0,T]}=(\xi^{\phi,0}_t)_{t\in [0,T]}$ be the solution of the stochastic differential equation
\begin{align}\label{eq:defxi0}
    \xi^0_t=\int_0^t (D\chi^\phi)^{-1}(s,\xi^0_s)\sigma dB_s
\end{align}
and define for $x,y,\in \R^n$ and $0\leq r<t\leq T,$ the function $G=G^\phi$ by 
\begin{align}\label{eq:transition}
    G(r,x,t,y)=\frac{det(D\chi(t,y))\exp\left(\gamma \Gamma(t,y)-\gamma\Gamma(r,x)-\frac{|\sigma^{-1}(\chi(t,y)-\chi(r,x))|^2}{2(t-r)}\right)}{det(\sigma)(2\pi (t-r))^{n/2}}.
\end{align}
Then, $G$ is the transition density of $\xi^0$ and $G^\phi$ only depends on $D^2\phi$ (and not on $D\phi(0)$ which would fully determine $\phi$). We denote $\mu^\phi$, the distribution whose density is $G^\phi(0,0,T,\cdot)$.
\end{lemma}

\begin{remark}
(i) The main assumption we make to establish the existence of equilibrium is based on \cite[Hypothesis 4]{cho} where the author postulates that in equilibrium both the pricing rule of the market maker and the strategy of the informed trader are functions of a novel state variable $(\xi_t)_{t\in[0,T]}$ that is a stochastic integral driven by $Y$ of the form $\xi_t=\int_0^t \lambda_s Y_s$ for a deterministic function $\lambda$.  In \cite{bose2020kyle}, it is shown that this novel state variable solves a stochastic differential equation driven by $Y$. 

In the equilibrium we construct, the strategy of the informed trader satisfies the inconspicuous trading property of \cite{cho}. Thus, under $\P$, the $\cF^Y$-distribution of $Y$ is the distribution of $\sigma B$. In this sense, the solution of \eqref{eq:defxi0} describes the distribution of the new state variable under $\cF^Y$ and in equilibrium. In fact, we only need $\xi^0$ through its distribution described with the transition density $G$ and not the exact value of the random variables $(\xi^0_t)$.

(ii) Through an application of the Ito's Lemma to $\Gamma(t,\xi_t)-\tilde v^\top\chi(t,\xi_t)$, the dynamics of $\Gamma,$ $\chi$, and the endogenously determined state variable $(\xi_t)$ allow us to simplify the dynamic problem of the informed trader \eqref{def:opt} into a problem of utility from terminal wealth. 
In this sense, the equations \eqref{eq:system1}-\eqref{eq:pdep}, allows us to pass from the dynamic problem of Kyle's model to a static problem of optimal transport at time $T$. 

\end{remark}

We now state the following theorem which is the main mathematical contribution of the paper and an extension of \cite[Theorem 3.1]{bose2020kyle} to the multi-dimensional case. 

\begin{theorem}\label{thm:fp}
Under Assumption \ref{assum:main}, there exists $\gamma_0>0$ so that for all $\gamma \in (0,\gamma_0)$ there exists $\phi\in C_l$ for some $l>0$ so that 
$(D_\xi\phi)_\sharp \mu^\phi=\nu$.
Equivalently $\phi$ is a convex solution of the Monge-Ampere equation
\begin{align}\label{eq:ma}
   G^\phi(0,0,T,x)=det(D_x^2\phi(x))f_\nu((D_x \phi)^\top(x)).
\end{align}
\end{theorem}

\begin{remark}
i) We note that the main contribution of Theorem \ref{thm:fp} is to establish the existence of a fixed-point. Indeed, for any $\phi\in C_l$, in Lemma \ref{lem:prep}, we define the function $\Gamma^\phi$ and $\chi^\phi$. Using these functions we can define $G^\phi$. Then, given $G^\phi$
one can find, thanks to the Brenier theorem \cite{b,m}, a convex function $\psi$ solving 
\begin{align}\label{eq:maaa}
   G^\phi(0,0,T,x)=det(D_x^2\psi(x))f_\nu((D_x \psi)^\top(x)).
\end{align}
Thus, we have defined a mapping from the set of functions $C_l$ to the set of all convex functions. Theorem \ref{thm:fp} states the existence of a fixed point for this mapping. 

(ii) The reason why we require this fixed-point comes from the fact that for such a fixed-point $(D\phi)^{-1}(\tilde v)$ (the inverse function of $D\phi:\R^n\mapsto \R^{1\times n}$) has distribution $ G^\phi(0,0,T,\cdot)$. Thus, the informed trader can find a trading strategy so that the family of random measures $G(t,\xi_t,T,\cdot)$ solves the filtering problem of the maker-maker of the unknown quantity $(D\phi)^{-1}(\tilde v)$. This crucial point simplifies the filtering problem of the market maker in equilibrium and allows us to pinpoint an equilibrium. 
\end{remark}

\section{Main result}\label{s:main}
We fix $l$ and $\phi\in C_l$ as constructed in Theorem \ref{thm:fp} and omit the dependence of different quantities in these variables. Thanks to the first order optimality condition in \eqref{eq:defchi},
we have 
$$ \chi(t,\xi)+(T-t)\gamma \sigma^2 (DE)^\top(t,\chi(t,\xi))=\xi. 
$$
Thus, $\chi$ is invertible in $\xi$ and for all $y\in C([0,T],\R^n)$, we can define the mapping $\xi :(t,y_\cdot)\in \Lambda\mapsto \R^n$ by the equality 
\begin{align}\label{eq:defxi1}
    \chi(t,\xi(t,y_\cdot))=\chi(0,0)+y_t+\int_0^t\gamma\sigma^2 P(s,\xi(s,y_\cdot))ds.
\end{align}
The following Lemma defines our candidate equilibrium strategy the informed trader and pricing rule for the market maker. 
\begin{lemma}\label{lem:add}
Under Assumption \ref{assum:main}, there exists $\gamma_1>0$ so that for all $\gamma\in (0,\gamma_1)$ and $\phi$ as in Theorem \ref{thm:fp}, the mapping $(t,y)\in \Lambda\mapsto \xi(t,y_\cdot)$ is $C^{1,2}(\Lambda)$ and the pricing rule for the market maker defined as 
\begin{align}\label{eq:pr}
   (t,y)\in \Lambda\mapsto H^*(t,y):=P^{\phi}(t,\xi(t,y_\cdot))=P(t,\xi(t,y_\cdot))
\end{align} 
is in $\cH$. The trading strategy for the informed trader defined as
\begin{align}\label{eq:ts}
    dX^*_t=\frac{(D\phi\textcolor{red}{^*})^{-1}(\tilde v)-\xi_t}{T-t}dt
\end{align}
where $\xi_t=\xi(t,Y_\cdot)$ is in $\cA(H^*)$. 
\end{lemma}

We are now ready to provide the main result of the paper. 
\begin{theorem}\label{thm:main}
Under Assumption \ref{assum:main}, the pair $(H^*,X^*)$ forms an equilibrium where the conditional distribution of $\xi_T$ conditionally on $\cF^Y_t$ is $G(t,\xi_t,T,\cdot)$ and the conditional distribution of $\tilde v$ is $\phi_\sharp G(t,\xi_t,T,\cdot)$\footnote{We abuse notation here by identifying distribution and density.}. In the equilibrium we construct, the expected utility of the informed trader conditional to his information is 
$$ -e^{-\gamma \left(\phi^c(\tilde v)-\frac{\gamma T}{2} |\sigma(\tilde v-P(0,0))|^2\right)-\gamma\Gamma(0,0)-\frac{\gamma^2 T}{2}P^\top(0,0) \sigma^2 P(0,0)}.$$
\end{theorem}

\begin{proof}[Proof of Theorem \ref{thm:main}]
The proof is provided in Subsections \ref{ss:add}, \ref{ss:mm}, and \ref{ss:informed trader}.
\end{proof}

\section{Properties of the Equilibrium and Examples}\label{s.ex}
\subsection{Properties of Equilibrium}
In the risk neutral case of \cite{cel}, it was shown that the expected wealth of the informed trader is $\phi^c(\tilde v)$. By convexity of $\phi^c$, large values of $\tilde v$ are advantageous for the informed trader. Indeed, these values correspond to unexpected or low probability events from the perspective of the market maker. 

With risk aversion, the part of the expected utility depending on $\tilde v$ is $$\phi^c(\tilde v)-\frac{\gamma T}{2} |\sigma(\tilde v-P(0,0))|^2.$$

Since $\phi^c$ and $\phi$ are smooth and convex conjugate of each other, we have the equalities
$$D\phi^c \left(D\phi(\xi)\right)=\xi\mbox{ and }D^2\phi^c \left(D\phi(\xi)\right)D^2\phi(\xi)=I_n.$$
Since by construction $0\leq D^2\phi\leq \frac{1}{\overline\lambda^2( \sigma)  \gamma T} I_n$, 
we have 
$$\underline \lambda(D^2\phi^c)\geq  \overline\lambda^2( \sigma)  \gamma T$$
and 
$$\tilde v\mapsto \phi^c(\tilde v)-\frac{\gamma T}{2} |\sigma(\tilde v-P(0,0))|^2$$
is a convex function. 
Thus, the profit of the informed trader is still large for large values of $\tilde v$. However, there is a loss of utility due to the fact that the informed trader is risk averse. 

\subsection{Gaussian Example}
We assume that $\nu$ is the Gaussian distribution $$N(m,\sigma_\nu^2)$$ where $\sigma_\nu$ is a symmetric positive matrix and $m\in \R^n$. Due to the linearity of transport maps between Gaussian distributions, we conjecture that $\phi$ constructed in Theorem \ref{thm:fp} is a quadratic function 
$$\phi(z)=\frac{1}{2}z^\top Az+B^\top z$$
for some $A\in \cS^n_+$ and $B\in \R^n$. Our objective is to use the fixed-point condition \eqref{eq:ma} to find $A,B$. To exhibit the fixed-point condition on $A$ and $B$ we conjecture that $E$ is a quadratic function 
$$E(t,z)=\frac{1}{2}z^\top A_t z+B_t^\top z+C_t $$
for some functions $A_t \in \cS^n_+$, $B_t\in \R^n$, $C_t\in \R$.

Injecting these in \eqref{eq:pdee}, we get
\begin{align*}
   &\frac{1}{2}z^\top \frac{\pa A_t}{\pa t}z+\frac{\pa B_t^\top}{\pa t}z+\frac{\pa C_t}{\pa t}+\frac{1}{2}tr(\sigma^2A_t)+\frac{\gamma}{2}\left(z^\top A_t\sigma +B_t^\top\sigma\right)\left(\sigma A_t^\top z+\sigma B_t\right)\\
   &\frac{1}{2}z^\top \left(\frac{\pa A_t}{\pa t}+{\gamma}A_t\sigma^2 A_t\right)z+\left(\frac{\pa B_t}{\pa t}+\gamma A_t\sigma^2 B_t\right)^\top z+\frac{\pa C_t}{\pa t}+\frac{1}{2}tr(\sigma^2A_t)+\frac{\gamma}{2}B_t^\top\sigma^2 B_t\\
   &=0.
\end{align*}
Using the final condition of the PDE of $E$, we have $A_T=A, B_T=B, C_T=0$ and 
\begin{align*}
    \frac{\pa A_t}{\pa t}+{\gamma}A_t\sigma^2 A_t&=0\\
    \frac{\pa B_t}{\pa t}+\gamma A_t\sigma^2 B_t&=0\\
    \frac{\pa C_t}{\pa t}+\frac{1}{2}tr(\sigma^2A_t)+\frac{\gamma}{2}B_t^\top\sigma^2 B_t&=0.
\end{align*}
Given its final condition, we have $A_t$ as
\begin{eqnarray*}
A_{t} &=&\left(  A^{-1}-{\gamma(T-t)\sigma ^{2}}%
\right) ^{-1} \\
&=&\frac{1}{\gamma}\sigma ^{-1}\left(   (\gamma\sigma A\sigma)
^{-1}-(T-t)I_n\right) ^{-1}\sigma ^{-1} .
\end{eqnarray*}

Similarly, we have that the solution of 
\begin{align*}
\dot{B}_{t}& =-\gamma A_{t}\sigma ^{2}B_{t} \\
B_{T}& = B 
\end{align*}
is
\begin{eqnarray*}
B_{t} &=A_tA^{-1} B.
\end{eqnarray*}
Finally, $C$ can be characterised as the solution of 
\begin{align*}
\dot{C}_{t}& =-\frac{1}{2}tr(\sigma^2A_t)-\frac{\gamma}{2}B_t^\top\sigma^2B_t\\
C_{T}& = 0 
\end{align*}
so that  
$$E(t,z)=\frac{1}{2}(z+A^{-1}B)^\top A_t(z+A^{-1}B)+C_t-\frac{1}{2}B^\top A^{-1}A_t A B.$$ 
Then, performing the minimization problem \eqref{eq:defchi}-\eqref{eq:defgamma} or using the first order optimality condition \eqref{eq:chip} and the Woodburry matrix identity,
\begin{align*}
    \left(I-\left[I-(\gamma (T-t)\sigma A\sigma)^{-1}\right]^{-1}\right)^{-1}&=I-\left(I-\left[I-(\gamma (T-t)\sigma A\sigma)^{-1}\right]\right)^{-1}\\
    &=I-\gamma (T-t)\sigma A\sigma
\end{align*}
we obtain 
\begin{align*}
    \chi(t,\xi)&=\sigma (I_n+\gamma(T-t)\sigma A_t\sigma )^{-1}\sigma^{-1}(\xi-\gamma(T-t)\sigma^2 A_t A^{-1}B)\\
    &= \sigma (I_n-\gamma(T-t)\sigma A \sigma)\sigma^{-1}(\xi-\gamma(T-t)\sigma^2 A_t A^{-1}B)\\
    &= \sigma (I_n-\gamma(T-t)\sigma A \sigma)\sigma^{-1}\xi- \gamma(T-t)\sigma^2 B\\
    \Gamma(t,\xi)&=E(t,\chi(t,\xi))\\
    P(t,\xi)&=A_t(\chi(t,\xi)+A^{-1}B)=A \xi+B.
\end{align*}
Thanks to Lemma \ref{lem:fp}, up to a normalization factor 
\begin{align}
    G(0,0,T,y)\propto e^{ -\frac{1}{2}\left(\frac{|\sigma^{-1}(y-\chi(0,0))|^2)}{T}-\gamma(y+A^{-1}B)^\top A(y+A^{-1}B)\right)}.
\end{align}
which is the Gaussian distribution 
$$\mu^\phi=N(M_{A,B},\Sigma^2_{A})$$
where $M_{A,B}=\left(\frac{\sigma^{-2}}{T}-\gamma A\right)^{-1}\left(\frac{\sigma^{-2}}{T}\chi(0,0)+\gamma  B\right)$ and $\Sigma_A=\left(\frac{\sigma^{-2}}{T}-\gamma A\right)^{-1/2}.$
The fixed point condition on $A$ and $B$ is that $D\phi(z)= Az+B$ pushes $\mu^\phi$ to $\nu=N(m,\sigma_\nu^2)$.
Recall that for any two non-degenerate Gaussian distributions $N(m_1, S_1^2)$ and $N( m_2,S_2^2)$, the Brenier map between these distributions is $$z\mapsto m_2+\Lambda(z-m_1).$$
where $\Lambda$ is the unique symmetric positive solution to $\Lambda S_1^2 \Lambda=S_2^2$. 

Thus, the fixed-point condition can be written as
$$A=\Lambda,\, B=m-\Lambda M_{A,B}$$
where $\Lambda \Sigma_A^2 \Lambda =\sigma_\nu^2$.
This yields 
\begin{align*}
    \frac{\sigma \sigma_\nu^2\sigma}{T} &=\sigma A\sigma \left(I-\gamma T\sigma  A\sigma \right)^{-1}\sigma A\sigma \\
    B&=m-A\left(\frac{\sigma^{-2}}{T}-\gamma A\right)^{-1}\left(\frac{\sigma^{-2}}{T}\chi(0,0)+\gamma  B\right).
\end{align*}
We compute $\chi(0,0)=-\gamma T\sigma^2 B$ so that $$B=m.$$
To compute $A$, we diagonalize $\frac{\sigma \sigma_\nu^2\sigma}{T}=U^\top D U$ for a symmetric positive matrix $D$ and orthongal matrix $U$. Then, $\tilde A=U\sigma A \sigma U^\top$ satisfies
$$ \tilde A\left(I_n-T\gamma \tilde A \right)^{-1}\tilde A=D.$$
Denoting by $d_i$ the diagonal terms of $D$, we have that $\tilde A$ is diagonal with diagonal terms $\tilde a_i$ satisfying
$$\frac{\tilde a_i^2}{1-T\gamma \tilde a_i}=d_i.$$
Thus,
$$\frac{1}{\tilde a_i}=\frac{T\gamma}{2}+\sqrt{\frac{T^2 \gamma^2}{4}+\frac{1}{d_i}}$$
and $\tilde A$ is determined, which intern yields to $A$ via the expression 
$$A=\sigma^{-1}U^\top \tilde A U\sigma^{-1}.$$
Finally we obtain $$P(t,\xi)=A \xi+m$$
and the dynamics of $\xi_t$ are given by 
$$d\xi_t=(D\chi)^{-1}(t,\xi_t)dY_t= \sigma (I_n-\gamma(T-t)\sigma A \sigma)^{-1}\sigma^{-1}dY_t$$

In one dimension, we have 
$$A=\frac{1}{\frac{T\sigma^2\gamma}{2}+\sqrt{\frac{T^2 \sigma^4\gamma^2}{4}+\frac{T\sigma^2}{\sigma_\nu^2}}}={\sqrt{\frac{ \sigma^4_\nu\gamma^2}{4}+\frac{\sigma_\nu^2}{T\sigma^2}}-\frac{\sigma_\nu^2\gamma}{2}}$$
which is denoted by $\lambda^*(1)$ in \cite{cho}. Thus, we recover the equilibrium computed in \cite{cho,bose2020kyle}.

For $\gamma=0$ and general distribution for $\nu$, for all $\phi$, $\mu^\phi$ is the Gaussian distribution with mean $\chi(0,0)=0$ and covariance matrix $T\sigma^2$ and $\xi_t=Y_t$ for all $t\in [0,T]$. Thus, $\phi$ is necessarily the Brenier's map from this given Gaussian distribution to $\nu$. Additionally \eqref{eq:system1}-\eqref{eq:pdep} are reducued to the same heat equation. The final condition for $P$ is the gradient of the final condition for $\Gamma$. Thus, for all $t$, $P(t,y)=(DE)^\top (t,y)$ with $\Gamma(T,y)$ being the Brenier's map from $N(0,T,\sigma^2)$ to $\nu$. This is the equilibrium described in \cite{cel}.

\section{Proofs}\label{s.proof}
\subsection{Proofs Results in Section \ref{s:prem}}
\begin{proof}[Proof of Lemma \ref{lem:prep}]
We fix $(t,z)\in [0,T)\times \R^n$. Note that $\exp(\gamma E)$ is defined via the expression
\begin{align}\label{eq:repE}
   &\E\left[\exp\left(\gamma \phi(z+\sigma(B_T-B_t))\right)\right]\\
   &=\frac{1}{\sqrt{2\pi (T-t})}\int \exp\left(\gamma\left( \phi(z+\sigma x)-\frac{| x|^2}{2\gamma (T-t)}\right)\right)dx.\notag
\end{align}
Thanks to convexity of $\phi$ and the multivariate Taylor expansion, for all $x\in \R^n$, we have that 
\begin{align*}
    &\gamma \phi(z)+ \gamma(D\phi)^\top (z) \sigma x-\frac{| x|^2}{2\gamma (T-t)}\leq \gamma \phi(z+\sigma x)-\frac{| x|^2}{2\gamma (T-t)}\\
    &\leq \gamma \phi(z)+ \gamma(D\phi)^\top (z) \sigma x+\frac{\overline\lambda^2( \sigma) \gamma l}{2}|x|^2-\frac{| x|^2}{2\gamma (T-t)}.
\end{align*}
Combining this inequality with $0\leq \frac{l\gamma \overline\lambda^2( \sigma)  (T-t)}{2}<\frac{l\gamma \overline\lambda^2( \sigma)  T}{2}<\frac{1}{2}$, the fact that $\phi$ has at most quadratic growth and $D\phi$ has at most linear growth, we obtain that $E$ is well defined and there exists a constant $C=C_{\gamma,l, T,|D\phi (0) |}$ so that  
$$ |E(t,z)|\leq C(|z|^2+1).$$

Similarly, by a dominated convergence theorem, we can obtain the wellposedness of the first two derivatives of $E$ in $z$ and for all $i,j=1,\ldots, n$, these derivatives admit the stochastic representations
\begin{align}\label{eq:defder}
    \frac{\pa E}{\pa z_j}(t,z)=\frac{\E\left[\frac{\pa \phi}{\pa z_j}(z+\sigma(B_T-B_t)) e^{\phi_{z,t}}\right]}{\E\left[e^{\phi_{z,t}}\right]}.
\end{align}
and\begin{align}\label{eq:defder2}
    \frac{\pa^2 E}{\pa z_j\pa z_i}(t,z)&=\frac{\E\left[\left(\frac{\pa^2 \phi}{\pa z_j\pa z_i}+\gamma \frac{\pa \phi}{\pa z_j}\frac{\pa \phi}{\pa z_i}\right)(z+\sigma(B_T-B_t)) e^{\phi_{z,t}}\right]}{\E\left[e^{\phi_{z,t}}\right]}\\
    &-\gamma\frac{\E\left[ \frac{\pa \phi}{\pa z_j}(z+\sigma(B_T-B_t))e^{\phi_{z,t}}\right]\E\left[\frac{\pa \phi}{\pa z_i}(z+\sigma(B_T-B_t))e^{\phi_{z,t}}\right]}{\E\left[e^{\phi_{z,t}}\right]^2}\notag
\end{align}
where $e^{\phi_{z,t}}=e^{\gamma \phi(z+\sigma(B_T-B_t))}$.

Thus, for any vector $v\in \R^n$ we have that 
\begin{align}\label{eq:cvxe}
   &v^\top D^2_z E(t,z)v={\tilde \E\left[v^\top D^2\phi(z+\sigma(B_T-B_t))v\right]}\\
   &\quad\quad+\gamma\left({\tilde \E\left[\ |D\phi(z+\sigma(B_T-B_t))v|^2\right]-|\tilde \E\left[ D\phi(z+\sigma(B_T-B_t))v\right]|^2}\right).\notag
\end{align}
where $\tilde \E$ is the expectation with respect to the probability distribution which is defined with the Radon-Nikodym derivative $ \frac{\exp\left(\gamma \phi_{z,t}\right)}{\E[\exp\left(\gamma \phi_{z,t}\right)]}$ which is integrable.
\eqref{eq:cvxe}, the convexity of $\phi$, and the Cauchy-Schwarz inequality give that $E$ is convex in $z$. 
Well-known Malliavin derivative representations of the derivative of conditional expectation (see \cite[Lemma 2.1]{fahim2011probabilistic}) also show that $E$ is 3 times continuously differentiable on $[0,T)$.
Since the derivatives in space are continuous and $E$ solves 
\begin{align}\label{eq:pdee}
\frac{\pa E}{\pa t}(t,z)+\frac{1}{2}tr\left(\sigma^2 D^2E(t,z)\right)+\frac{\gamma}{2}| DE(t,z)\sigma|^2&=0\\
E(T,z)&=\phi(z).
\end{align}
we have that $E$ is $C^{1,3}([0,T),\R^n)$. $E$ is also $C^{0}([0,T],\R^n)$ by its stochastic representation and the dominated convergence theorem. 

The convexity of $E$ combined with the strong convexity of $z\mapsto \frac{1}{2\gamma (T-t)}|\sigma^{-1}(\xi-z)|^2$ gives the strong convexity of 
$$(t,\xi,z)\mapsto E(t,z)+\frac{1}{2\gamma(T-t)}|\sigma^{-1}(\xi-z)|^2.$$
Thus, $\chi(t,\xi)$ is unique, well-defined, and characterized via the first order optimality condition 
\begin{align}\label{eq:focp1}\gamma(T-t)\sigma^2(D E)^\top(t,\chi(t,\xi))+\chi(t,\xi)=\xi\end{align}
for all $t\in [0,T)$ and $\chi(T,\xi)=\xi$. 
We define
\begin{align}\label{eq:defr}
R(t,z):=\gamma(T-t)\sigma^2(D E)^\top(t,z)+z
\end{align}
so that by \eqref{eq:focp1} we have 
\begin{align}
R(t,\chi(t,\xi))&=\xi.\label{eq:R_chi_defn}
\end{align}
The implicit functions theorem shows that the inverse of $R$, which is $\chi$, is continuously differentiable and
\begin{align}
(D \chi)^{-1}(t,\xi)=\gamma(T-t)\sigma^2D^2 E(t,\chi(t,\xi))+I_n.\label{eq:invchi}
\end{align}
Thus, $\chi\in C^{1,2}([0,T),\R^n)\cup C^{0}([0,T],\R^n)$. By the definition of $\Gamma$ and $P$, these functions are also in $C^{1,2}([0,T),\R^n)\cup C^{0}([0,T],\R^n)$.

We now show that the functions defined by \eqref{eq:defchi}-\eqref{eq:defgamma} satisfy \eqref{eq:system1}-\eqref{eq:system2}. 
Differentiating \eqref{eq:R_chi_defn} in time once and in space twice with respect to $\xi_i$ and $\xi_j$, for all $k=1,\ldots,n$ we obtain
\begin{align*}
&\frac{\pa R^k}{\pa t}(t,\chi(t,\xi))+\ \sum_{l=1}^{n} \frac{\pa R^k}{\pa z_l}(t,\chi(t,\xi))\frac{\pa \chi^l}{\pa t}(t,\xi) = 0\\
&\sum_{l,m=1}^{n} \frac{\pa^2 R^k}{\pa z_l \pa z_m}(t,\chi(t,\xi))\frac{\pa \chi^m}{\pa \xi_i}(t,\xi)\frac{\pa \chi^l}{\pa \xi_j}(t,\xi) +\sum_{l=1}^{n} \frac{\pa R^k}{\pa z_l}(t,\chi(t,\xi))\frac{\pa^2 \chi^l}{\pa \xi_j\xi_i}(t,\xi)= 0.
\end{align*}
We multiply the second equality with $\left(\sigma^\top (D\chi)^{-\top}(t,\xi)\right)_{r,i}\left((D\chi)^{-1}(t,\xi)\sigma\right)_{j,r}$ and sum in $r,i,j$ to obtain for all $k=1,\ldots,n$ the equality
\begin{align}\label{eq:r_chi}
     \sum_{l=1}^{n} \frac{\pa R^k}{\pa z_l}(t,\chi(t,\xi))\left(\frac{\pa \chi^l}{\pa t}(t,\xi)+\frac{1}{2}tr\Big(\sigma^\top (D\chi)^{-\top}(t,\xi)D^2 \chi^{l}(t,\xi)(D\chi)^{-1}(t,\xi)\sigma\Big)\right)\notag\\
+ \frac{\pa R^k}{\pa t}(t,\chi(t,\xi))+\frac{1}{2}tr\Big(\sigma^\top D^2_z R^{k}(t,\chi(t,\xi))\sigma\Big)=0.
\end{align}
Differentiating \eqref{eq:defr} in time once we obtain 
\begin{align*}
     \frac{\pa R^k}{\pa t}(t,z)&=-\gamma\left(\sigma^{2} (D E)^\top(t, z)\right)^k + \gamma (T-t)\Big(\sigma^2\frac{\partial (D E)^\top}{\partial t}(t,z)\Big)^k.
\end{align*}
Differentiating \eqref{eq:pdee} and using \eqref{eq:defr}, we have 
\begin{align*}
    \frac{1}{2}tr\Big(\sigma^\top D^2 R^{k}(t,z)\sigma\Big)&=-\gamma (T-t)\big(\sigma^2\frac{\pa (D E)^\top}{\pa t}(t,z)\big)^k\\
    &-\gamma^2(T-t)\Big(\sigma^2D^2E(t,z)\sigma^2(D E)^\top(t,z)\Big)^k
\end{align*}
Thus, injecting this equality and the expression for $\frac{\pa R^k}{\pa t}(t,z)$ in \eqref{eq:r_chi}, we have
\begin{align*}
   &0=-\gamma\left[\gamma(T-t)\left(\sigma^{2} D^{2} E(t, \chi(t,\xi)) \sigma^{2} (D E)^\top(t,\chi(t,\xi))\right)^{k}+\left(\sigma^{2} (D E)^\top(t,\chi(t,\xi))\right)^{k}\right]\\
   &+ \sum_{l=1}^{n} \frac{\pa R^k}{\pa z_l}(t,\chi(t,\xi))\left(\frac{\pa \chi^l}{\pa t}(t,\xi)+\frac{1}{2}tr\Big(\sigma^\top (D\chi)^{-\top}(t,\xi)D^2 \chi^{l}(t,\xi)(D\chi)^{-1}(t,\xi)\sigma\Big)\right) .
\end{align*}
Using the definition of $P(t,\xi)$ given by equation \eqref{eq:defp} we have,
\begin{align*}
    &0=-\gamma\left[\gamma(T-t)\left(\sigma^{2} D^{2}_z E(t, \chi(t,\xi)) \sigma^{2} P(t,\xi)\right)^{k}+\left(\sigma^{2} P(t,\xi)\right)^{k}\right]\\
    &+ \sum_{l=1}^{n} \frac{\pa R^k}{\pa z_l}(t,\chi(t,\xi))\left(\frac{\pa \chi^l}{\pa t}(t,\xi)+\frac{1}{2}tr\Big(\sigma^\top (D\chi)^{-\top}(t,\xi)D^2 \chi^{l}(t,\xi)(D\chi)^{-1}(t,\xi)\sigma\Big)\right)\\
    &=-\gamma\left[\sum_{l=1}^{n}\Big(\gamma(T-t)\sigma^{2} D^{2}_z E(t, \chi(t,\xi))+Id\Big)_{k,l}\Big( \sigma^{2} P(t,\xi)\Big)^{l}\right]\\
    &+ \sum_{l=1}^{n} \frac{\pa R^k}{\pa z_l}(t,\chi(t,\xi))\left(\frac{\pa \chi^l}{\pa t}(t,\xi)+\frac{1}{2}tr\Big(\sigma^\top (D\chi)^{-\top}(t,\xi)D^2 \chi^{l}(t,\xi)(D\chi)^{-1}(t,\xi)\sigma\Big)\right)\\
    &= \sum_{l=1}^{n} \frac{\pa R^k}{\pa z_l}(t,\chi(t,\xi))\Big(-\gamma\sigma^{2} P(t,\xi)\Big)^l\\
    &+ \sum_{l=1}^{n} \frac{\pa R^k}{\pa z_l}(t,\chi(t,\xi))\left(\frac{\pa \chi^l}{\pa t}(t,\xi)+\frac{1}{2}tr\Big(\sigma^\top (D\chi)^{-\top}(t,\xi)D^2\chi^{l}(t,\xi)(D\chi)^{-1}(t,\xi)\sigma\Big)\right).
\end{align*}
Finally we obtain that for all $k$,
\begin{align*}
   0&=  \sum_{l=1}^{n} \frac{\pa R^k}{\pa z_l}(t,\chi(t,\xi))\\
   &\quad\quad\times\left(\frac{\pa \chi^l}{\pa t}(t,\xi)+\frac{1}{2}tr\Big(\sigma^\top (D\chi)^{-\top}(t,\xi)D^2 \chi^{l}(t,\xi)(D\chi)^{-1}(t,\xi)\sigma\Big)-(\gamma\sigma^2 P(t,\xi))^l\right)
\end{align*}
Multiplying with the inverse of the matrix $ \frac{\pa R^k}{\pa z_l}(t,\chi(t,\xi))$, which by the help of equations \eqref{eq:R_chi_defn} and \eqref{eq:invchi} is $D \chi(t,\xi),$ we obtain that for all $l= 1,\ldots,n$, the equality
\begin{align*}
   0=  \frac{\pa \chi^l}{\pa t}(t,\xi)+\frac{1}{2}tr\Big(\sigma^\top (D\chi)^{-\top}(t,\xi)D^2 \chi^{l}(t,\xi)(D\chi)^{-1}(t,\xi)\sigma\Big)-(\gamma\sigma^2 P(t,\xi))^l
\end{align*}
which is \eqref{eq:system2}.
Next, we prove equation \eqref{eq:system1}. Using the definition of $\Gam$ given by \eqref{eq:defgamma} and differentiating it once in time and twice in space with respect to $\xi_i$ and $\xi_j$, we get
\begin{align}
    \frac{\pa \Gam}{\pa t}(t,\xi)&=\frac{\pa E}{\pa t}(t,\chi(t,\xi))+ \sum_{l=1}^{n} \frac{\pa E}{\pa z_l}(t,\chi(t,\xi))\frac{\pa \chi^l}{\pa t}(t,\xi)\label{eq:tgamma}\\
    \frac{\pa^2 \Gam}{\pa \xi_i \xi_j}(t,\xi)&=\sum_{l,m=1}^{n} \frac{\pa^2 E}{\pa z_l \pa z_m}(t,\chi(t,\xi))\frac{\pa \chi^m}{\pa \xi_i}(t,\xi)\frac{\pa \chi^l}{\pa \xi_j}(t,\xi)\notag \\
    &+\sum_{l=1}^{n} \frac{\pa E}{\pa z_l}(t,\chi(t,\xi))\frac{\pa^2 \chi^l}{\pa \xi_j\xi_i}(t,\xi) \label{eq:gamma2}  .
\end{align}
Similarly to \eqref{eq:r_chi}, using \eqref{eq:gamma2}, we form the trace term, which is given by
\begin{align*}
  &tr\Big(\sigma^\top (D\chi)^{-\top}(t,\xi)D^2 \Gamma(t,\xi)(D\chi)^{-1}(t,\xi)\sigma\Big)=tr(\sigma^\top D^2_zE(t,\chi(t,\xi)\sigma))\\
  &\quad\quad\quad\quad\quad+\sum_{l=1}^{n} \frac{\pa E}{\pa z_l}(t,\chi(t,\xi))tr\Big(\sigma^\top (D\chi)^{-\top}(t,\xi)D^2 \chi^{l}(t,\xi)(D\chi)^{-1}(t,\xi)\sigma\Big)
\end{align*}
Combining this equality with \eqref{eq:tgamma} and using \eqref{eq:pdee}, \eqref{eq:system2}, and \eqref{eq:defp}, we have
\begin{align*}
     &\frac{\pa \Gamma}{\pa t}(t,\xi)+\frac{1}{2}tr\Big(\sigma^\top (D\chi)^{-\top}(t,\xi)D^2 \Gamma(t,\xi)(D\chi)^{-1}(t,\xi)\sigma\Big)\\
     &=-\frac{\gamma}{2}| DE(t,\chi(t,\xi))\sigma|^2\\ 
     &+\sum_{l=1}^{n} \frac{\pa E}{\pa z_l}(t,\chi(t,\xi))\bigg( \frac{\pa \chi^l}{\pa t}(t,\xi)+\frac{1}{2}tr\Big(\sigma^\top (D\chi)^{-\top}(t,\xi)D^2 \chi^{l}(t,\xi)(D\chi)^{-1}(t,\xi)\sigma\Big)\bigg)\\
     &=-\frac{\gamma}{2}| DE(t,\chi(t,\xi))\sigma|^2 + \gamma DE(t,\chi(t,\xi))\sigma^2P(t,\xi)\\
     &=-\frac{\gamma}{2}|\sigma P(t,\xi)|^2+\gamma|\sigma P(t,\xi)|^2=\frac{\gamma}{2}|\sigma P(t,\xi)|^2.
\end{align*}

In order to obtain \eqref{eq:pdep}, we write the optimality condition \eqref{eq:focp1} as 
$$\gamma(T-t)\sigma^2P(t,\xi)+\chi(t,\xi)=\xi$$
and use \eqref{eq:system2}.

It remains to show the continuous dependence of $\Gamma,\chi$ in $\phi$. Let $\phi_k,\phi\in C_l$ be such that $\phi_k\to \phi$ uniformly on compact sets of $\R^n$.
Denote $E^k=E^{\phi_k}$ and fix $(t,\xi)\in [0,T)\times \R^n$.
Thanks to \eqref{eq:cvxe}, the functions $z\mapsto E^k(t,z)$ are convex and given the representation \eqref{eq:repE} and the subsequent bounds on the integrands, uniform convergence on compact sets of $\phi^k$ to $\phi$ easily implies that $z\mapsto E^k$ converges pointwise to $z\mapsto E(t,z)$. By an application of \cite[Theorem 10.8]{rockafellar2015convex}, this convergence is uniform on compact sets of $\R^n$. 

The sequence of functions
$$z\in \R^n\mapsto E^k(t,z)+\frac{1}{2\gamma(T-t)}|\sigma^{-1}(\xi-z)|^2$$
are strongly convex uniformly in $k$. This is also the case for 
$$z\in \R^n\mapsto E(t,z)+\frac{1}{2\gamma(T-t)}|\sigma^{-1}(\xi-z)|^2.$$
Due to uniform strong convexity of this function and the uniform convergence on compact sets, one can easily construct a ball a round $\chi(t,\xi)$ so that $\chi^k(t,\xi)$ is in this ball for $k$ large enough. Thus, the sequence $(\chi^k(t,\xi))_k$ admits a limit. Due to the uniqueness of the minimizer of 
$$z\in \R^n\mapsto E(t,z)+\frac{1}{2\gamma(T-t)}|\sigma^{-1}(\xi-z)|^2$$
we have $\chi^k(t,\xi)\to \chi(t,\xi)$. This convergence combined with the uniform convergence on compact sets of $E^k$ yields to $\Gamma^k(t,\xi)=E^k(t,\chi^k(t,\xi))\to \Gamma(t,\xi)$ which completes the proof. 
\end{proof}
\begin{proof}[Proof of Lemma \ref{lem:fp}] Fix $r\in [0,T)$ and for $t\in [r,T]$, denote 
\begin{align}\label{eq:tildey}\tilde Y_{r,t}=\chi(r,x)+\sigma(B_t-B_r)\mbox{ and }\tilde Z_{r,t}=e^{\gamma E(t,\tilde Y_{r,t})-\gamma E(t,\chi(r,x))}.\end{align}
Thanks to \eqref{eq:pdee}, we have  
\begin{align}
\label{eq:defpsi}    \tilde Z_{r,t}=e^{-\frac{1}{2}\int_r^t \gamma^2 |\sigma P(s,\chi^{-1}(s,\tilde Y_{r,s}))|^2 ds+\int_r^t \gamma P^\top(s,\chi^{-1}(s,\tilde Y_{r,s}))\sigma dB_s}.\notag
\end{align}
We now define the process $\chi^0_t=\chi(t,\xi^0_t)$ which satisfies the dynamics
$$d\chi^0_t =\sigma dB_t+\gamma \sigma^2 P(t,\chi^{-1}(t,\chi_t^0))dt$$
with $\chi^0_r=\chi(r,x)$.
Thanks to \eqref{eq:defpsi} and \cite[Theorem 3.1]{tt}, the density of $\chi_t^0$
is given by 
$$\frac{1}{det(\sigma)(2\pi(t-r))^{n/2}} e^{-\frac{|y-\chi(r,x)|^2}{2(t-r)}}\E[\tilde Z_{r,t}|\tilde Y_{r,t}=y].$$
Thanks to \eqref{eq:tildey}, this density is 
$$\frac{1}{det(\sigma)(2\pi(t-r))^{n/2}} e^{-\frac{|y-\chi(r,x)|^2}{2(t-r)}}e^{\gamma E(t,y)-\gamma E(t,\chi(r,x))}.$$
Thus, the density of $\xi^0_t=\chi^{-1}(t,\chi_t^0)$ is \eqref{eq:transition}.

It remains to show that $\mu^\phi$ only depends on $D^2\phi$. Note that since $\phi(0)=0$ for all $\phi\in C_l$, due to the equality $\phi(z)=\int_0^1z^\top D\phi(sz)ds$ we directly have that $\mu^\phi$ only depends on $D\phi$. Additionally, the equality   
$$\phi(z)=\int_0^1 \int_0^s z^\top D^2 \phi(rz)zdr ds+ z^\top D\phi(0).$$ 
shows that $\mu^\phi$ only depends on $D^2\phi$ if and only if for all $A\in \R^n$, 
$$\mu^{\phi_A}=\mu^\phi$$ 
where $\phi_A(z)=\phi(z)+A^\top z$ (this means that shifting the derivative of $\phi$ by a constant does not change $\mu^\phi$). 
In order to prove this claim we fix $A\in \R^n$ and denote 
$$(E^A,\Gamma^A,\chi^A,P^A,G^A):=(E^{\phi_A},\Gamma^{\phi_A},\chi^{\phi_A},P^{\phi_A},G^{\phi_A}).$$
Thanks to \eqref{eq:repE}, 
\begin{align*}
   e^{\gamma E^A(t,z)}&=\frac{e^{A^\top z}}{(2\pi (T-t))^{n/2}}\int e^{\gamma\left( \phi(z+\sigma x)+A^\top \sigma x-\frac{| x|^2}{2\gamma (T-t)}\right)}dx\\
   &=\frac{e^{\gamma A^\top z+\frac{\gamma (T-t)|\sigma A|^2}{2}}}{(2\pi (T-t))^{n/2}}\int e^{\gamma\phi(z+\sigma x)-\frac{1}{2 (T-t)}| x-\gamma (T-t) \sigma A|^2}dx\\
   &=e^{\gamma A^\top z+\frac{\gamma (T-t)|\sigma A|^2}{2}+\gamma E(t,z+\gamma(T-t)\sigma^2A)}.
\end{align*}
Thus, by direct computation and using \eqref{eq:chip} and the definitions of $\chi^A, \Gamma^A$, we have
\begin{align*}
    E^A(t,z)&=A^\top z+\frac{ (T-t)|\sigma A|^2}{2}+ E(t,z+\gamma(T-t)\sigma^2A)\\
    D E^A(t,z)&=A^\top + DE(t,z+\gamma(T-t)\sigma^2A)\\
    \chi^A(t,\xi)&=\chi(t,\xi)-\gamma (T-t)\sigma^2 A\\
    \Gamma^A(t,\xi)&=E^A(t,\chi^A(t,\xi))\\
    &=A^\top \chi(t,\xi)-\frac{ (T-t)|\sigma A|^2}{2}+ \Gamma(t,\xi)
\end{align*}
Thus, using the definition of $G$, we obtain the invariance
\begin{align*}
    G^A(r,x,t,y)&=\frac{det(D\chi^A(t,y))\exp\left(\gamma \Gamma^A(t,y)-\gamma\Gamma^A(r,x)-\frac{|\sigma^{-1}(\chi^A(t,y)-\chi^A(r,x))|^2}{2(t-r)}\right)}{det(\sigma)(2\pi (t-r))^{n/2}}\\
    &=\frac{det(D\chi(t,y))\exp\left(\gamma \Gamma(t,y)-\gamma\Gamma(r,x)-\frac{|\sigma^{-1}(\chi(t,y)-\chi(r,x))|^2}{2(t-r)}\right)}{det(\sigma)(2\pi (t-r))^{n/2}}\\
    &=G(r,x,t,y).
\end{align*}
which concludes the proof. 
\end{proof}

\begin{proof}[Proof of Theorem \ref{thm:fp}]
The proof is based on the proof of \cite[Theorem 3.1]{bose2020kyle}. We only provide a detailed proof of the technicalities that are due to the multi-dimensional aspects. Given $\kappa>0$ defined in Assumption \ref{assum:main}, define $\gamma_0=\frac{\sqrt{\kappa}}{2\overline \lambda(\sigma)\sqrt{T}}$ and $l=\frac{1}{\overline \lambda(\sigma)\sqrt{\kappa T}}$ so that for all $\gamma\in (0,\gamma_0)$, we have that 
$$0<l=\frac{1}{\overline \lambda(\sigma)\sqrt{T\kappa}}=\frac{1}{2\overline \lambda^2(\sigma)T\gamma_0}<\frac{1}{\overline \lambda^2(\sigma)T\gamma}.$$
Thus, we can use  Lemmas \ref{lem:prep}-\ref{lem:fp} and the Brenier's theorem \cite{b,m} to define a mapping 
$$\cM:C_l\mapsto C_\infty$$
that associates to $\phi\in  C_l$, $\cM(\phi)$ which is the Brenier map pushing $\mu^\phi$ (which has density $G^\phi(0,0,T,\cdot)$ to $\nu$ and normalized to be $0$ at $0$. 
Note that 
\begin{align*}
    G(0,0,T,y)&=\frac{\exp\left(\gamma \phi(y)-\gamma\Gamma(0,0)-\frac{|\sigma^{-1}(y-\chi(0,0))|^2}{2T}\right)}{det(\sigma)(2\pi T)^{n/2}}\mbox{ and}\\
- v^\top D^2_y(\ln G(0,0,T,y))v&=\frac{|\sigma^{-1}v|^2}{ T}-\gamma v^\top D^2\phi(y)v\leq \frac{|v|^2}{\overline \lambda^2(\sigma) T}
\end{align*}
for all $v\in \R^n$. 
Thus, \cite[Theorem 6.1]{ko} implies that 
$$\kappa  (\overline \lambda (D^2 \cM(\phi)))^2 \leq \frac{1}{\overline \lambda^2(\sigma) T}.$$
Equivalently 
$$\overline \lambda (D^2 \cM(\phi))\leq \frac{1}{\sqrt{\overline \lambda^2(\sigma) T\kappa}}.$$
Due to our choice of $l$, we have that 
$$\overline \lambda (D^2 \cM(\phi))\leq \frac{1}{\sqrt{\overline \lambda^2(\sigma) T\kappa}}=l.$$

Thus, $\cM(\phi)\in C_l$ and $\cM$ is indeed a mapping from $C_l$ to itself. 

We endow $C_l$ with uniform convergence on compact sets of derivatives, i.e., 
$\phi_k\in C_l\to \phi\in C_l$ if 
$\sup_{x\in K}|D\phi_k(x)-D\phi(x)|\to 0$ for all compact $K\subset \R^n$.
We now show that $\cM$ is continuous in this topology. Fix $\phi_k\in C_l\to \phi\in C_l$. Note that the uniform convergence on compact sets of the derivatives imply the uniform convergence on compact sets of the functions. By Lemma \ref{lem:prep}, 
$(\Gamma^{\phi_k}(0,0),\chi^{\phi_k}(0,0))\to(\Gamma^{\phi}(0,0),\chi^{\phi}(0,0)). $
Thus, we have the uniform convergence on compact sets
\begin{align} \label{eq:cvg}
G^{\phi_k}(0,0,T,y)=\frac{\exp\left(\gamma \phi_k(y)-\gamma\Gamma^{\phi_k}(0,0)-\frac{|\sigma^{-1}(y-\chi^{\phi_k}(0,0))|^2}{2T}\right)}{det(\sigma)(2\pi T)^{n/2}}\\
\to G^{\phi}(0,0,T,y)=\frac{\exp\left(\gamma \phi(y)-\gamma\Gamma^{\phi}(0,0)-\frac{|\sigma^{-1}(y-\chi^{\phi}(0,0))|^2}{2T}\right)}{det(\sigma)(2\pi T)^{n/2}}.\notag
\end{align}
Note that $G^\phi$ only depends on $D^2\phi$ and $\phi\in C_l$.
By the choice of $l$ and simple estimates we also have that 
\begin{align}\label{eq:boundg}
G^{\phi_k}(0,0,T,y)\vee G^{\phi}(0,0,T,y)\leq C\overline G(y)
\end{align}
where $\overline G(y)$ is the density of some non-degenerate Gaussian distribution and $C>0$. 

We now show that $D\cM(\phi_k)\to D\cM(\phi)$ in probability with respect to the density $\overline G$. \eqref{eq:cvg}-\eqref{eq:boundg} easily implies that $\mu^{\phi_k}\to \mu^\phi$ for the weak convergence of probability measures. Similarly to the proof of \cite[Corollary 23]{v}, this convergence implies the convergence of transport plans, meaning for all continuous bounded function $f$ we have the convergence 
$$\int f(y, D\cM(\phi_k)(y))G^{\phi_k}(0,0,T,y)dy\to \int f(y, D\cM(\phi)(y))G^{\phi}(0,0,T,y)dy.$$
Thanks to the continuity of $\cM(\phi)$, we take $f(y,z)=1\wedge \frac{|D\cM(\phi)(y)-z|}{\e}$ to obtain from this convergence that 
$$\lim_{k\to \infty}\int 1\wedge \frac{|D\cM(\phi)(y)- D\cM(\phi_k)(y)|}{\e}G^{\phi_k}(0,0,T,y)dy= 0.$$
Thus, we have the convergence in measure
$$\lim_{k\to \infty}\int 1_{\{|D\cM(\phi)(y)- D\cM(\phi_k)(y)|\geq \e \}}\overline G(y)dy= 0.$$
Since the function $D\cM(\phi_k)$ is uniformly Lipschitz continuous, their convergence in measure easily implies their pointwise convergence. This in return implies the uniform convergence on compact sets. This is precisely the convergence in $C_l$ of $\cM(\phi_k)$ to $\cM(\phi)$. 
Thus, $\cM:C_l\mapsto C_l$ is a continuous mapping. 

We now show that $\cM(C_l)$ is compact which is sufficient to have a fixed-point. Let $\phi_k$ be a sequence in $C_l$. Define $\tilde \phi_k(z)=\phi_k(z)- D\phi_k(0)z$. 
Thanks to \eqref{lem:fp}, $G^{\phi_k}=G^{\tilde \phi_k}$ and therefore $\cM(\phi_k)=\cM(\tilde \phi_k)$. 
Thus, without loss of generality we can assume that $D\phi_k(0)=0$. This property combined with the uniform Lipschitz continuity of $D\phi_k$ and Arzela-Ascoli theorem yield that there is a subsequence of $D\phi_k$ denoted $D\phi_{k_n}$ converging to some $\psi$ uniformly on compact sets of $\R^n$. By the uniform convergence, $\psi$ is a conservative vector field and there exists $\phi\in C_l$ so that $\psi=D\phi$.
We can now use the continuity of $\cM$ to obtain the fact that $\{\cM(\phi_k):k\geq 1\}$ has an accumulation point. Thus, by the Schauder fixed-point theorem, there exists $\phi\in C_l$ so that 
$\cM(\phi)=\phi$. 

\end{proof}

\subsection{Admissibility of candidate equilibrium strategies}\label{ss:add}
\begin{proof}[Proof of Lemma \ref{lem:add}]
The definition of $(\xi_t)$ as \eqref{eq:defxi1} and the regularity of $\chi$ directly implies that for all $i=1,\ldots n$, $\xi^i\in C^{1,2}(\Lambda)$ and by composition with a smooth function $(t,y_\cdot)\mapsto H^*(t,y_\cdot)=P(t,\xi(t,y_\cdot))$ is in $C^{1,2}(\Lambda)$. 

Given our definition, $H^*(T,\sigma B_\cdot)$ has the same distribution as $P(T,\xi_T^0)=(D\phi)^\top(\xi_T^0)$. Since $D\phi$ is Lipschitz continous we easily have the integrability. Similarly $H^*(t,\sigma B_\cdot)$ has the same distribution as $P(t,\xi_t^0)$. Due to the martingality of $P(t,\xi_t^0)$ we easily have the integrability condition of Definition \ref{def:pricing}. 

In order to show the admissibility of the strategy \eqref{eq:ts}, we only need to check the Novikov condition \eqref{eq:addx} for $\gamma$ small enough. 
The proof of this statment is similar to \cite[Section 5.2]{bose2020kyle}.
\end{proof}

\subsection{Market maker's problem}\label{ss:mm}
If the informed trader uses the candidate strategy \eqref{eq:ts} then $\xi$ satisfies
\begin{align}
    d\xi_t= (D\chi)^{-\top}(t,\xi_t)\frac{(D\phi)^{-1}(\tilde v)-\xi_t}{T-t}dt+(D\chi)^{-1}(t,\xi_t)\sigma dB_t
\end{align}
 where $(D\phi)^{-1}(\tilde v)$ is unknown by the market maker.
By the choice of $\phi$, $(D\phi)^{-1}(\tilde v)$ has distribution $G(0,0,T,\cdot)$. 
Thanks to \cite{liptser1977statistics}, the filtering equation implies that the density $\rho(t,\cdot)$ of $(D\phi)^{-1}(\tilde v)$ conditional to $\cF^Y_t=\cF^\xi_t$ is the unique solution to 
 \begin{align}\label{eq:fili}
     \rho(0,y)&=G(0,0,T,y)\\
     \frac{d\rho(t,y)}{\rho(t,y)}&=\left(\frac{y-\int z\rho(t,z)dz}{T-t}\right)^\top\sigma^{-2}\left(D\chi(t,\xi_t)d\xi_t- \frac{\int z\rho(t,z)dz-\xi_t}{T-t}dt\right).\label{eq:fild}
 \end{align}
 Since $G$ is the transition density of $\xi^0$, $G(t,\xi^0_t,T,y)$ is a martingale on $[0,T)$ and $(t,\xi)\in [0,T)\times\R^n\to G(t,\xi,T,y)$ solves \eqref{eq:pdep}. Thus, applying Ito's formula we obtain $$dG(t,\xi_t,T,y)= D_\xi G(t,\xi_t,T,y)d\xi_t.$$
 Note that 
\begin{align*}
    \frac{D_\xi G(t,\xi,T,y)}{G(t,\xi,T,y)}&= -\gamma  \left(D\Gamma(t,\xi)(D\chi)^{-1}(t,\xi)+\frac{(\chi(t,\xi)-y)^\top \sigma^{-2}}{\gamma(T-t)}\right)D\chi(t,\xi)\\
    &= -\gamma  \left(P^\top (t,\xi)+\frac{(\chi(t,\xi)-y)^\top\sigma^{-2}}{\gamma(T-t)}\right)D\chi(t,\xi)\\
    &= -\gamma\frac{(\xi-y)^\top\sigma^{-2}}{\gamma(T-t)} D\chi(t,\xi)
\end{align*}
 where we used \eqref{eq:chip} to obtain the last equality. 
 Thus, 
 $$\frac{dG(t,\xi_t,T,y)}{G(t,\xi_t,T,y)}=\frac{(y-\xi_t)^\top}{(T-t)}\sigma^{-2}D\chi(t,\xi_t)d\xi_t.$$
 Additionally, by martinality of $\xi^0$, we have 
 \begin{align}\label{eq:mean1}\int G(t,\xi_t,T,y)y dy=\xi_t.\end{align}
 Thus, the density $G(t,\xi_t,T,\cdot)$ solves \eqref{eq:fild}. It also satisfies the initial condition \eqref{eq:fili} and we have that 
$G(t,\xi_t,T,\cdot)=\rho(t,\cdot)$.

Note that $G(T,\xi_T,T,y)$ is a Dirac mass at $\xi_T$. Thus, $\xi_T=(D\phi)^{-\top}(\tilde v)$ almost surely. 
\eqref{eq:mean1} also implies that
$\xi_t=\E[\xi_T|\cF^Y_t] =\E[(D\phi)^{-\top}(\tilde v)|\cF^Y_t]$ which means that $(\xi_t)$ is a $\cF^Y$ martingale. Finally, by \eqref{eq:pdep} an application of Ito's Lemma $P(t,\xi_t)$ is a martingale satisfying $P(T,\xi_T)= D\phi^\top((D\phi)^{-\top}(\tilde v))=\tilde v$. Thus, $P(t,\xi_t)= \E[\tilde v|\cF^Y_t]$. This proves the rationality of the pricing rule $H$ defined in \eqref{eq:pr}. 

Note the proof provided here is self contained and does not rely on construction of Markov bridges as in \cite{ccd}. 
\subsection{Informed trader's problem}\label{ss:informed trader}

We now assume that the market maker uses the pricing rule \eqref{eq:pr}
and we fix $X\in \cA(H^*)$ an admissible strategy for the informed trader which has the semimartingale decomposition
$dX_t=\theta_tdt+\a_t dB_t$, where $\a_t$ is a $n\times n$ symmetric matrix valued $\cF$-progressively measurable process. If the informed trader uses this strategy $Y$ then satisfies 
$$dY_t=\theta_tdt+(\sigma+\a_t) dB_t.$$

$\xi$ is defined as the solution of \eqref{eq:defxi1}. By conjecturing its dynamics as 
\begin{align}\label{eq:dxi}
d\xi_t&=(D_\xi\chi)^{-1}(t,\xi_t)\left(dY_t-\eta_t dt\right)
\end{align}
and applying Ito's formula to \eqref{eq:defxi1}, 
we can easily identify that 
\begin{align*}
\eta_t^i =&\frac{1}{2}tr\left(\a_t(D_\xi\chi)^{-\top}(t,\xi_t)D^2_\xi\chi^{i}(t,\xi_t)(D_\xi\chi)^{-1}(t,\xi_t)\a_t\right)\notag\\
&+\frac{1}{2}tr\left(\sigma(D_\xi\chi)^{-\top}(t,\xi_t)D^2_\xi\chi^{i}(t,\xi_t)(D_\xi\chi)^{-1}(t,\xi_t)\a_t\right)\notag\\
&+\frac{1}{2}tr\left(\a_t(D_\xi\chi)^{-\top}(t,\xi_t)D^2_\xi\chi^{i}(t,\xi_t)(D_\xi\chi)^{-1}(t,\xi_t)\sigma\right)\notag.
\end{align*}
By a direct computation we also have 
\begin{align}\label{eq:cov1}
    \gamma \sum_{i=1}^nd\langle P^{i},X^i\rangle_t&=\gamma tr\left(\a_tD_\xi P(t,\xi_t) (D_\xi\chi)^{-1}(t,\xi_t)(\a_t+\sigma)\right)  dt\\
        \notag&=\gamma tr\left(\a_t D^2_\xi E(t,\chi(t,\xi_t))(\a_t+\sigma)\right)  dt.
\end{align}
Using the dynamics of $\xi_t$, the multidimensional Ito's formula yields
\begin{align*}
 \gamma d(\Gamma(t,\xi_t))&=\Bigg\{
\frac{\gamma^{2}}{2} P_t^{\top} \sigma^{2} P_t+\gamma D_{\xi} \Gamma(t, \xi_t)\left(D_{\xi} \chi\right)^{-1}(t,\xi_t)
\left(\theta_{t}-\eta_{t}\right)\notag\\   
&+\frac{\gamma}{2} tr\Big[\alpha_t^{\top}(D_{\xi} \chi)^{-\top}(t, \xi_t) D_{\xi}^{2} \Gamma(t, \xi_t)
\left(D_{\xi}\chi\right)^{-1}(t, \xi_t) \sigma\notag\\
&+(\sigma+\alpha_t)^{\top}\left(D_{\xi} \chi\right)^{-\top}(t, \xi_t)D_{\xi}^{2} \Gamma(t, \xi_t)
\left(D_{\xi} \chi\right)^{-1}(t, \xi_t) \alpha_t\Big] \Bigg\}dt \notag\\
&+\gamma D_{\xi} \Gamma(t, \xi_t)\left(D_{\xi} \chi\right)^{-1}(t, \xi_t) (\sigma+\alpha_t) d B_{t}\notag\\
&=\Big[\frac{\gamma^{2}}{2} P_t^{\top} \sigma^{2} P_t+\gamma P_t^{\top}
\left(\theta_{t}-\eta_{t}\right)\Big]dt+\gamma P_t^{\top} (\sigma+\alpha_t)dB_t\notag\\
&+\frac{\gamma}{2} tr\Big[\alpha_t^{\top}(D_{\xi} \chi)^{-\top}(t, \xi_t) D_{\xi}^{2} \Gamma(t, \xi_t)
\left(D_{\xi}\chi\right)^{-1}(t, \xi_t) \sigma\notag\\
&+(\sigma+\alpha_t)^{\top}\left(D_{\xi} \chi\right)^{-\top}(t,\xi_t)D_{\xi}^{2} \Gamma(t, \xi_t)
\left(D_{\xi} \chi\right)^{-1}(t, \xi_t) \alpha_t\Big]dt
\end{align*}
where we use $D_{\xi} \Gamma(t, \xi)=DE(t,\chi(t,\xi))D_{\xi} \chi(t,\xi)=P^\top(t,\xi)D_{\xi} \chi(t,\xi)$ to get the second equality. Differentiating $D_{\xi} \Gamma(t, \xi)$ one more time and using the the expression of $\eta$, we obtain that 
\begin{align}\label{eq:dgamma}
 \gamma d(\Gamma(t,\xi_t))&=\Big[\frac{\gamma^{2}}{2} P_t^{\top} \sigma^{2} P_t+\gamma P_t^{\top}
\theta_{t}\Big]dt+\gamma P_t^{\top} (\sigma+\alpha_t)dB_t\notag\\
&+\frac{\gamma}{2} tr\Big[\alpha_t^{\top} D^2_\xi E(t,\chi(t,\xi_t))\sigma+(\sigma+\alpha_t)^{\top} D^2_\xi E(t,\chi(t,\xi_t))\alpha_t\Big]dt.
\end{align}

Similarly, we have
\begin{align}\label{eq:dvchi}
  &-\gamma d(\tilde v^\top\chi(t,\xi_t))=\sum_{i=1}^{n}-\gamma d\left(\tilde v^i\chi^{i}(t,\xi_t)\right)\notag\\
  &=(-\gamma^2 \tilde v^{\top}  \sigma ^{2}P_t-\gamma\tilde v^{\top}\theta_{t}) d t-\gamma \tilde v^{\top}(\sigma+\alpha_t) d B_{t}
  \end{align}
Combining \eqref{eq:cov1}, \eqref{eq:dgamma}, and \eqref{eq:dvchi}, we get the following decomposition of the wealth
\begin{align*}
-\gamma W_T=&-\gamma \int_0^T  (\tilde v-P_t)^{\top}dX_t+ \gamma \sum_{i=1}^n\langle P^{i},X^i\rangle_T
\\
=&\frac{T\gamma^{2}}{2} \tilde v^{\top} \sigma^{2} \tilde v-\gamma \left(\tilde v^{\top} \xi_T -\phi(\xi_T)\right)+\gamma(\tilde v^\top \chi(0,0)- \Gamma(0,0))\\
&+\int_0^T\frac{\gamma}{2} tr\Big[\alpha_t^{\top} D^2_\xi E(t,\chi(t,\xi_t))\alpha_t\Big]dt\\
&-\int_0^T \frac{\gamma^{2}}{2} |\sigma (P_t-\tilde v)|^{2}dt -\gamma \int_0^T(P_t-\tilde v)^{\top} \sigma dB_t.
\end{align*}

Thus, finally we obtain
\begin{align}
   \E\left[-e^{ -\gamma W_T}|\cF_0\right]&=\tilde \E\left[-e^{-\gamma ( \tilde v^\top \xi_T-\phi(\xi_T))+\int_0^T\frac{\gamma}{2}tr\Big(\alpha^2_tD_{z}^2 E\left(t,\chi(t,\xi_t)\right)dt\Big)}|\cF_0\right]\notag\\
&\times e^{\gamma(\tilde v^\top \chi(0,0)-\Gamma(0,0))+\frac{\gamma^2 T}{2} \tilde v^\top \sigma^2 \tilde v}
   \end{align}
   where $\tilde \E$ is obtain by an application of Girsanov's theorem and under the associated probability 
   $$dB_t-\gamma \sigma (\tilde v-P_t)dt$$ 
   defines a $\cF$ Brownian motion. Note that this is indeed an equivalent change of measure by the definition of admissibility in Definition \ref{def:strategy}. 
   
   Since $D_{z}^2 E$ and $\a^2$ are symmetric non-negative matrices, their product only has non-negative eigenvalues and therefore $tr\Big(\alpha^2_tD_{z}^2 E\left(t,\chi(t,\xi_t)\right)\Big)\geq 0$. 
   Additionally, $\tilde v^\top \xi_T-\phi(\xi_T)\leq \phi^c(\tilde v):=\sup \{ \tilde v^{\top}y-\phi(y)|y\in \R^n\}$. 
   Thus, 
   $$ \E\left[-e^{ -\gamma W_T}|\cF_0\right]\leq -e^{-\gamma \phi^c(\tilde v)+\gamma(\tilde v^\top \chi(0,0)-\Gamma(0,0))+\frac{\gamma^2 T}{2} \tilde v^\top \sigma^2 \tilde v}$$
   and any strategy of the informed trader with $\a=0$ and $D\phi(\xi_T)=\tilde v$ achieves this upper bound and is therefore optimal. As shown in Subsection \ref{ss:mm}, these properties hold the strategy \eqref{eq:ts}.

\bibliographystyle{siamplain}
\bibliography{ref}     
\end{document}